\documentclass{amsart}
\usepackage{graphicx} 
\usepackage[utf8]{inputenc}
\usepackage{amsmath}
\usepackage{amsthm}
\usepackage{amsfonts}
\usepackage{amssymb}
\usepackage{cancel}
\usepackage[shortlabels]{enumitem}
\usepackage{mathrsfs}
\usepackage{tikz, tikz-cd}
\usepackage[colorlinks=true,linkcolor=blue,citecolor=blue,urlcolor=blue,citebordercolor={0 0 1},urlbordercolor={0 0 1},linkbordercolor={0 0 1}]{hyperref} 
\usepackage[nameinlink]{cleveref}
\usepackage{biblatex}
\theoremstyle{plain}
\newtheorem{thm}{Theorem}[section]
\newtheorem{cor}[thm]{Corollary}
\newtheorem{lem}[thm]{Lemma}

\newtheorem{prop}[thm]{Proposition}

\newtheorem*{thm*}{Theorem}

\newtheorem{thmx}{Theorem}

\theoremstyle{definition}
\newtheorem{rem}[thm]{Remark}

\newtheorem{defn}[thm]{Definition}

\newcommand{\Db}{\text{D}^b}

\title{Twisted Derived Equivalences Between Abelian Varieties}
\author{Tyler Lane}

\begin{document}
\begin{abstract}
    We generalize a result of Popa-Schnell and show that the isogeny class of the Picard variety is twisted derived invariant.  Using this, we prove that any twisted Fourier-Mukai partner of an abelian variety is an abelian variety.  We then provide a necessary and sufficient isogeny-based condition for two abelian varieties to be twisted derived equivalent.
\end{abstract}
\maketitle

\section*{Introduction}
Here we study the twisted Fourier-Mukai partners of complex abelian varieties.  An important derived invariant of a smooth projective variety $X$ is the group $R^0_X:=\text{Aut}^0_X \times \text{Pic}^0_X$ (see \cite{Ros09}, \cite{Rou11}).  A natural replacement for $R^0_X$ in the twisted setting was introduced by Olsson in \cite{Ols25}.  It enables us to generalize several known results and arguments to the twisted case.  For example, arguing as in \cite{PS11}, we can prove

\begin{thmx}[{\Cref{t}}]\label{i: t PS}
    Let $X$ and $Y$ be smooth projective varieties over $\mathbb{C}$.  If $X$ and $Y$ are twisted derived equivalent, then $\text{Pic}^0_X$ and $\text{Pic}^0_Y$ are isogenous.
\end{thmx}
Arguing as in \cite{Ros09}, we can then prove

\begin{thmx}[{\Cref{t2}}]\label{i: c AV}
    Let $X$ and $Y$ be smooth projective varieties over $\mathbb{C}$, and suppose that $X$ and $Y$ are twisted derived equivalent.  If $X$ is an abelian variety, then so is $Y$.  Moreover, $Y$ is isogenous to $X$.
\end{thmx}

This result enables us to restrict our attention to twisted derived equivalences between abelian varieties.  Our main result is the following.
\begin{thmx}[{\Cref{t criterion}, \Cref{4: t main}}]\label{i t: main}
    Let $X$ and $Y$ be complex abelian varieties.  The following are equivalent:
    \begin{enumerate}
        \item $X$ and $Y$ are twisted derived equivalent
        \item There is an isogeny $f:\hat{X} \to Y$ whose kernel is isomorphic to $\oplus_{i=1}^r (\mathbb{Z}/m_i\mathbb{Z})^2$.
    \end{enumerate}
\end{thmx}
We say a few words about the proof.  Given an isogeny $f:\hat{X} \to Y$ as in \Cref{i t: main}, we can use \cite{Bri13} to produce an irreducible homogeneous projective bundle $P$ over $X$.  We show that $P$ is the projectivization of a simple semi-homogeneous twisted vector bundle bundle $E$ on $X$, and that $\text{ker}(f)$ is the stabilizer of $E$ for the $\text{Pic}^0_X$ action given by the tensor product.  It follows that the orbit map for the $\text{Pic}^0_X$ action on $E$ factors through a map from $Y$ to the moduli space of twisted sheaves on $X$.  To show that the corresponding family on $Y \times X$ is the kernel of a twisted Fourier-Mukai transform, we establish a twisted analogue of \cite[Lemma 4.8]{Or02} (see \Cref{c orthog}).

From here it is not difficult to show that the existence of an isogeny $f$ as in \Cref{i t: main} is equivalent to the existence of a twisted derived equivalence between $X$ and $Y$ with vector bundle kernel (see \Cref{3: c main}).  Removing the assumptions on the kernel is much more difficult, and it requires \cite[Theorem 1.1]{LLT25}.  \cite[Theorem 1.1]{LLT25} enables us to use Polishchuk's work on symplectic abelian varieties to produce the desired result; in particular, our result will follow from a modified version of \cite[Lemma 2.2.7(ii)]{Pol12}.

\subsection*{Relationship with other work} While preparing this manuscript, the results of \cite{LLT25} and \cite{Li25} were announced.  The work completed prior to \cite{LLT25} and \cite{Li25} makes up Sections 1-3.  In particular, our proofs of \Cref{i: t PS}, \Cref{i: c AV}, and \Cref{3: c main} are independent of \cite{LLT25} and \cite{Li25}.  In order to strengthen \Cref{3: c main} and complete the proof of \Cref{i t: main}, we must rely on \cite[Theorem 1.1]{LLT25} which we were not able to prove ourselves.

There is some overlap between our results and those of \cite{LLT25} and \cite{Li25}.  Using different methods, a generalization of \Cref{i: c AV} was proven in \cite[Theorem 1.9]{Li25}, and the $(2) \Rightarrow (1)$ implication of \Cref{i t: main} was proven in \cite[Corollary 7.5]{LLT25}.  Semi-homogeneous twisted vector bundles were also studied in \cite[Section 3]{Li25}.

Necessary and sufficient conditions for twisted derived equivalences between abelian varieties were also given in \cite[Theorem 1.1]{LLT25} and \cite[Proposition 6.8, Theorem 6.9]{Li25}.  The conditions given in \Cref{i t: main} are new, and we believe they are also the most concrete.  

\subsection*{Notation and Conventions} We always work over the field $k=\mathbb{C}$.  For a smooth projective variety $X$, we let $\text{Aut}_X$ (resp.\ $\text{Pic}_X$) denote the Picard scheme (resp. group scheme of automorphisms) of $X$, and we denote their neutral components by $\text{Aut}^0_X$ (resp. $\text{Pic}^0_X$).  We will typically write $\text{Pic}^0(X)$ (resp.\ $\text{Aut}^0(X))$ for the group of $k$-points of $\text{Pic}^0_X$ (resp.\ $\text{Aut}^0_X$).  We only consider $k$-points of schemes and thus write $t \in T$ instead of $t \in T(k)$.
We assume that the reader is familiar with twisted derived categories (see \cite{BS21}, \cite{C00}, \cite{Lan24} for an introduction).

\subsection*{Acknowledgements} I would like to thank Brendan Hassett, Zhiyuan Li, Ziwei Lu, Zhichao Tang, and Kenny Blakey for helpful comments on earlier drafts.
\section{Preliminaries}
In this section we recall the necessary preliminaries from \cite{Ols25}.  Let $X$ be a smooth projective variety, and let $\pi:\mathscr{X} \to X$ be a $\mathbb{G}_m$-gerbe.  Let $\alpha \in H^2(X,\mathbb{G}_m)$ the cohomology class of $\mathscr{X}$.

Let $\mathscr{A}ut_{\mathscr{X}}$ denote the fibered category that associates to a scheme
$T$ the groupoid of isomorphisms $\mathscr{X}_T \to \mathscr{X}_T$ that induce the identity map $\mathbb{G}_m \to \mathbb{G}_m$ on stabilizer groups.  It was shown in \cite[Theorem 1.1]{Ols25} that $\mathscr{A}ut_{\mathscr{X}}$ is a $\mathbb{G}_m$-gerbe over a group algebraic space $\text{Aut}_{\mathscr{X}}$ whose neutral component fits into an exact sequence
$$
  0 \to \text{Pic}^0_{X} \xrightarrow{\eta_X} \text{Aut}^0_{\mathscr{X}} \xrightarrow{q_X} \text{Aut}^0_{X}\to 0,$$
  where $q_X$ sends an automorphism $\sigma:\mathscr{X} \to \mathscr{X}$ to its rigidification.  Note that this implies that $\text{Aut}^0_\mathscr{X}$ is a group scheme.  The key result we need is that $\text{Aut}^0_{\mathscr{X}}$ is a twisted derived invariant:

  \begin{thm}[{\cite[Theorem 1.9]{Ols25}}]\label{t: Olsson}
  Let $\mathscr{Y} \to Y$ be a $\mathbb{G}_m$-gerbe over a smooth projective variety $Y$, and let $\beta \in H^2(Y,\mathbb{G}_m)$ denote the cohomology class of $\mathscr{Y}$.  Let $\Phi_{\mathscr{E}}:\text{D}^b(X,\alpha) \to \text{D}^b(Y,\beta)$ be a Fourier-Mukai equivalence.  Then $\Phi_{\mathscr{E}}$ induces an isomorphism $$\Psi:\text{Aut}^0_{\mathscr{X}} \xrightarrow{\sim} \text{Aut}^0_\mathscr{Y}.$$
\end{thm}
Let $\text{Aut}^0(\mathscr{X})=\text{Aut}^0_\mathscr{X}(k)$, let $\sigma \in \text{Aut}^0(\mathscr{X})$, and suppose that $\Psi(\sigma)=\tau$.  Then by \cite[Remark 7.2]{Ols25}, we have that $\Phi_{\mathscr{E}} \circ \sigma^* \simeq \tau^* \circ \Phi_{\mathscr{E}}$.  Note that we have natural isomorphisms
$$\Phi_\mathscr{E} \circ \sigma^*=Rp_{2*}(\mathscr{E}\otimes^Lp_1^*\sigma^*(-)) \simeq Rp_{2*}(\mathscr{E}\otimes^L(\sigma \times \text{id})^*p_1^*(-))$$
$$\simeq Rp_{2*}((\sigma \times \text{id})^*(((\sigma^{-1} \times \text{id})^*\mathscr{E}) \otimes^L p_1^*(-))) \simeq  Rp_{2*}((\sigma^{-1} \times \text{id})_*(((\sigma^{-1} \times \text{id})^*\mathscr{E}) \otimes^L p_1^*(-)))$$
$$\simeq Rp_{2*}(((\sigma^{-1}\times \text{id})^*\mathscr{E})\otimes^Lp_1^*(-))$$
and
$$\tau^*\circ \Phi_{\mathscr{E}}=\tau^*Rp_{2*}(\mathscr{E}\otimes^Lp_1^*(-))\simeq Rp_{2*}((\text{id} \times \tau)^*(\mathscr{E} \otimes^L p_1^*(-))) \simeq Rp_2^*(((\text{id}\times \tau)^*\mathscr{E})\otimes^Lp_1^*(-)).$$
By uniqueness of Fourier-Mukai kernels (see \cite[Theorem 1.1]{CS07}), we obtain the following.

\begin{lem}\label{l: kernels}
Let $\Phi_{\mathscr{E}}:\text{D}^b(X,\alpha) \xrightarrow{\sim} \text{D}^b(Y,\beta)$ be a Fourier-Mukai equivalence.  Let $\Psi$ be as in \Cref{t: Olsson}, let $\sigma \in \text{Aut}^0({X})$, and let $\Psi(\sigma)=\tau$.  Then $(\sigma \times \tau)^*\mathscr{E} \simeq \mathscr{E}$.
\end{lem}
We are especially interested in automorphisms $\sigma \in \text{Aut}^0(X)$ which lie in the image of $\eta_X:\text{Pic}^0_X \to \text{Aut}^0_{\mathscr{X}}$.
\begin{prop}[{\cite[Proposition 6.1]{Ols25}}] \label{p: tensor}
Let $L \in \text{Pic}^0(X)$, and let $\eta_X(L):\mathscr{X} \to \mathscr{X}$ denote the corresponding automorphism.  Then there is a natural isomorphism between the functors $\eta(L)^*:\text{D}^b(X,\alpha) \to \text{D}^b(X,\alpha)$ and $(-)\otimes L:\text{D}^b(X,\alpha) \to \text{D}^b(X,\alpha)$.
\end{prop}
\begin{cor}\label{c: line bundles}
Let $\Phi_{\mathscr{E}}:\text{D}^b(X,\alpha) \xrightarrow{\sim} \text{D}^b(Y,\beta)$ be a Fourier-Mukai equivalence.  Let $\Psi$ be as in \Cref{t: Olsson}, let $L \in \text{Pic}^0(X)$, and suppose that there exists $M \in \text{Pic}^0(Y)$ such that $\Psi(\eta_X(L))=\eta_Y(M)$.  Then $\mathscr{E}\otimes(L\boxtimes M^{-1})\simeq \mathscr{E}.$
\end{cor}
\begin{proof}
  By assumption and by \Cref{p: tensor}, we have a natural isomorphism  $\Phi_{\mathscr{E}}((-)\otimes L) \simeq \Phi_{\mathscr{E}}(-)\otimes M.$  Using the natural isomoprhisms
  $$\Phi_\mathscr{E}((-) \otimes L) \simeq \Phi_{\mathscr{E}\otimes p_1^*L}, \quad \Phi_{\mathscr{E}}(-)\otimes M \simeq \Phi_{\mathscr{E}\otimes p_2^*M}$$
  together with uniqueness of kernels we obtain the desired result.
\end{proof}
\section{Invariance of the Picard Variety}
In this section we prove a twisted analogue of \cite[Theorem A(1)]{PS11}
\begin{thm}\label{t}
Let $X$ and $Y$ be smooth projective varieties, and let $$\Phi_{\mathscr{E}}:\text{D}^b(X,\alpha)\xrightarrow{\sim}\text{D}^b(Y,\beta)$$
be a Fourier-Mukai equivalence.  Then $\text{Pic}^0_X$ and $\text{Pic}^0_Y$ are isogenous. 
\end{thm}
\begin{proof}
    
Let $\mathscr{X} \to X$ (resp. $\mathscr{Y} \to Y$) be the $\mathbb{G}_m$-gerbe whose cohomology class is $\alpha \in H^2(X,\mathbb{G}_m)$ (resp. $\beta \in H^2(Y,\mathbb{G}_m$)).  Recall that we have a diagram
$$\begin{tikzcd}
{0} \arrow[r] & {\text{Pic}^0_X} \arrow[r,"\eta_X"] & {\text{Aut}^0_\mathscr{X}} \arrow[r,"q_X"] \arrow[d,"\Psi"] & {\text{Aut}^0_X} \arrow[r] & {0} \\
  {0} \arrow[r] & {\text{Pic}^0_Y} \arrow[r,"\eta_Y"] & {\text{Aut}^0_\mathscr{Y}} \arrow[r,"q_Y"] & {\text{Aut}^0_Y} \arrow[r] & {0}
\end{tikzcd}$$
where the horizontal sequences are exact, and the vertical arrow is the isomorphism $\Psi:\text{Aut}^0_\mathscr{X} \xrightarrow{\sim} \text{Aut}^0_\mathscr{Y}$ defined in \Cref{t: Olsson}.  Let $B \subseteq \text{Aut}^0_Y$ denote the image of $\text{Pic}^0_X$ under $\mathfrak{b}:=q_{Y} \circ \Psi$, and let $A \subseteq \text{Aut}^0_X$ denote the image of $\text{Pic}^0_Y$ under $\mathfrak{a}:=q_X \circ \Psi^{-1}$.  Note that $A$ and $B$ are abelian varieties.  If both $A$ and $B$ are trivial, then it is immediate that $\text{Pic}^0_{X} \simeq \text{Pic}^0_Y$, so we exclude this case from now on.
Let $H_\mathscr{X}=\text{Aut}^0_\mathscr{X}\times_{\text{Aut}^0_X} A$.  $H_{\mathscr{X}}$, being an extension of $A$ by an abelian variety, is thus an abelian variety.  It follows that we may split the short exact sequence
$$0 \xrightarrow{} \text{Pic}^0_X \xrightarrow{\eta_X} H_\mathscr{X} \xrightarrow{q_A} A \xrightarrow{} 0$$
in the isogeny category, i.e., that there exists a morphism $s: A \to H_{\mathscr{X}}$ such that $q_A \circ s = [n]$ for some nonzero integer $n$.  In particular, there exists an isogeny
$$\lambda:=(s,\eta_{X}):A \times \text{Pic}^0_X \to H_{\mathscr{X}}.$$
We define $H_{\mathscr{Y}}$ similarly and note that $\Psi$ induces an isomorphism $H_{\mathscr{X}} \simeq H_{\mathscr{Y}}$.

Let the abelian variety $A \times B$ act on $X \times Y$ by automorphisms.  Take a point $(x,y)$ contained in $\text{Supp}(\mathscr{E})$ and consider the orbit map
$$f=(f_1,f_2):A \times B \to X \times Y,\quad (a,b) \mapsto (a(x),b(y)).$$

By \cite[Lemma 2.2, Theorem 2.3]{PS11}, the induced map $A \times B \to \text{Alb}_X \times \text{Alb}_Y$ has finite kernel, so the dual map $f^*:\text{Pic}^0_X \times \text{Pic}^0_Y \to \hat{A} \times \hat{B}$ is surjective.  In particular, the map $f_1^*:\text{Pic}^0_X \to \hat{A}$ is surjective.

We define morphisms
$$\kappa:H_\mathscr{X} \to A\times B,\quad \phi \mapsto (q_X(\phi),q_Y(\Psi(\phi)))$$
and
$$\varepsilon: A \times \text{Pic}^0_X \to A \times B \times \hat{A},\quad (a, L) \mapsto (\kappa(\lambda(a,L)),f_1^*L)$$
which fit into a commuative diagram
$$\begin{tikzcd}
{A \times \text{Pic}^0_X} \arrow[r, "\varepsilon"] \arrow[d,swap, "\lambda"] & {A \times B \times \hat{A}} \arrow[d, "p_{12}"] \\
{H_\mathscr{X}} \arrow[r,swap, "\kappa"]           & {A \times B}          
\end{tikzcd}$$
By definition of $A$ and $B$, $\kappa$ is surjective, so that
$$\text{dim}(A)+\text{dim}(\text{Pic}^0_X) = \text{dim}(A)+\text{dim}(B)+\text{dim}(\text{ker}(p_{12}\circ \varepsilon)).$$
The projection onto $\text{Pic}^0_X$ induces a morphism $\text{ker}(\varepsilon) \to \text{ker}(f_1^*)$ whose kernel is finite; indeed, if $(a,0) \in \text{ker}(\varepsilon)$, then $a$ is contained in $\text{ker}(q_X \circ s)$, which is finite.  This fact, along with surjectivity of $f_1^*$ implies that
$$\text{dim}(\text{ker}(\varepsilon)) \leq \text{dim}(\text{ker}(f_1^*))=\text{dim}(\text{Pic}^0_X)-\text{dim}(A).$$
If the restriction of $p_{12}$ to $\text{im}(\varepsilon)$ has finite kernel, then we must have that $\text{dim}(\text{ker}(\varepsilon))=\text{dim}(\text{ker}(p_{12}\circ \varepsilon))$ so that $\text{dim}(A) \leq \text{dim}(B)$ by the above inequalities.

Let $\mathscr{Z} \to A \times B$ denote the pullback of the gerbe $\mathscr{X} \times \mathscr{Y} \to X \times Y$ along $f$, let $\gamma \in H^2(A \times B,\mathbb{G}_m)$ denote the cohomology class of $\mathscr{Z}$, and let $\mathscr{F}:=Lf^*\mathscr{E} \in \text{D}^b(A \times B ,\gamma)$.

Let $(0,0,P) \in \text{im}(\varepsilon)$, and let $(a,L) \in A \times \text{Pic}^0_X$ be such that $\varepsilon(a,L)=(0,0,P)$.  Then we must have that $s(a)\in \text{Pic}^0_X(X)$.  Let $s(a)=M \in \text{Pic}^0(X)$, and note that $M$ is contained in the finite set $s(A) \cap \text{Pic}^0_X$.  By assumption, we must also have that $\Psi(\lambda(a,L))=\Psi(\eta_{X}(L \otimes M))$ is contained in $ \text{ker}(q_Y)=\text{Pic}^0_Y$, so we can write $\Psi(\eta_{X}(L \otimes M)) =\eta_Y(N)$ for some line bundle $N \in \text{Pic}^0(Y)$.  By \Cref{c: line bundles}, this implies that
$$\mathscr{E} \otimes ((L \otimes M) \boxtimes N^{-1} )\simeq \mathscr{E}$$
so that
$$ \mathcal{H}^i(\mathscr{F})\otimes ((P \otimes f_1^*M)\boxtimes f_2^* N^{-1}) \simeq \mathcal{H}^i(\mathscr{F}).$$
Now, let
$$\Sigma(\mathcal{H}^i(\mathscr{F}))=\{\mathscr{L} \in \text{Pic}^0(A \times B) : \mathcal{H}^i(\mathscr{F})\otimes \mathscr{L} \simeq \mathcal{H}^i(\mathscr{F})\}.$$
We claim that $\Sigma(\mathcal{H}^i(\mathscr{F}))$ is finite.  To prove this, it suffices to show that $\mathcal{H}^i(\mathscr{F})$ is a vector bundle.  We we will do so using the following lemma.

\begin{lem}\label{l: vector bundle}
  Let $X$ be an abelian variety, let $\pi:\mathscr{X} \to X$ be a $\mathbb{G}_m^2$-gerbe, and let $\mathscr{G}$ be a $(-1,1)$-twisted coherent sheaf on $\mathscr{X}$.  If for all $x \in X$ there exists an automorphism $\sigma:\mathscr{X} \to \mathscr{X}$ such that 
  \begin{enumerate}
  \item $\sigma$ induces the identity on stabilizers,
  \item the rigidification of $\sigma$ is the map $t_x:X \to X$,
  \item $\sigma^*\mathscr{G} \simeq \mathscr{F}$,
  \end{enumerate}
  then $\mathscr{G}$ is a vector bundle.
\end{lem}
\begin{proof}
  Let $f:Y \to X$ be an isogeny trivializing $\mathscr{X}$, and let $\mathscr{Y}$ denote the fiber product
  $$
\begin{tikzcd}
{\mathscr{Y}} \arrow[r,, "g"] \arrow[d, swap, "p"] & {\mathscr{X}} \arrow[d, "\pi"] \\
{Y} \arrow[r, swap, "f"]           & {X}          
\end{tikzcd}.$$
It suffices to show that $\mathscr{H}:=g^*\mathscr{G}$ is a vector bundle.  Fix $y \in Y$.  There is a lift $\tau:\mathscr{Y} \to \mathscr{Y}$ of $t_y$ such that $\tau$ and $\mathscr{H}$ satisfy conditions (1)-(3) above: if we let $x=\pi(y)$, we can take $\sigma$ as in the lemma and let $\tau=f^*\sigma$.  Since $\mathscr{Y}$ is trivial, there exists a coherent sheaf $H$ on $Y$ such that $\mathscr{H} \simeq p^*H \otimes \mathscr{L}$ for some $(-1,1)$-twisted invertible sheaf $\mathscr{L}$.  Since $\tau$ induces the identity on stabilizers, we have that $\tau^*\mathscr{L} \simeq \mathscr{L} \otimes p^*L$ for some line bundle $L$ on $Y$.  Since $\tau^*\mathscr{H} \simeq \mathscr{H}$, we have that $p^*(t_y^*H \otimes L)\otimes \mathscr{L} \simeq p^*H \otimes \mathscr{L}$, which implies that $t_y^*H \otimes L \simeq H$.  By \cite[(3.4)]{dJO22}, $H$ is a vector bundle, and thus $\mathscr{H}$ is too.
\end{proof}
Let $(a,b) \in A \times B$.  Recall that $\mathscr{Z}=f^*(\mathscr{X} \times \mathscr{Y})$, that $\mathscr{F}=Lf^*\mathscr{E}$, and that there is a Cartesian square
$$\begin{tikzcd}
{A \times B} \arrow[r, "t_{(a,b)}"] \arrow[d,swap, "f"] & {A \times B} \arrow[d, "f"] \\
{X \times Y} \arrow[r, swap, "{(a,b)}"]           & {X \times Y}          
\end{tikzcd} $$

To produce $\sigma$ as in \Cref{l: vector bundle}, it suffices to find an automorphism $\tau:\mathscr{X}\times \mathscr{Y}\to \mathscr{X} \times \mathscr{Y}$ such that
      \begin{enumerate}
      \item $\tau$ induces the identity on stabilizers,
      \item the rigidification of $\tau$ is $(a,b)$,
      \item $\tau^*\mathscr{E} \simeq \mathscr{E}$.
      \end{enumerate}
Then we can take $\sigma=f^*\tau$.
Recall that the morphism $\kappa:H_{\mathscr{X}} \to A \times B$ is surjective.  It follows from the definition of $\kappa$ that there exists a pair of isomorphisms $(\phi,\psi) \in \text{Aut}^0({\mathscr{X}})\times \text{Aut}^0(\mathscr{Y})$ such that $\phi$ and $\psi$ induce the identity of on inertia, the rigidificiation of $\phi$ (resp. $\psi$) is $a$ (resp. $b$), and such that $\Psi(\phi)=\psi$.  It follows from \Cref{l: kernels} that $(\phi,\psi)^*\mathscr{E} \simeq \mathscr{E}$.  Thus the morphism $\tau=(\phi,\psi):\mathscr{X} \times \mathscr{Y} \to \mathscr{X} \times \mathscr{Y}$ has the desired properties.

We have shown that $\mathcal{H}^i(\mathscr{F})$ is a vector bundle so that $\Sigma(\mathcal{H}^i(\mathscr{F}))$ is a finite set.  In particular, we have that $P \otimes  f_1^* M$ is contained in the finite set  $\tilde{\Sigma}=p_{\hat{A}}(\Sigma(\mathcal{H}^i(\mathscr{F})))$, where $p_{\hat{A}}:\hat{A} \times \hat{B} \to \hat{A}$ is the projection.  Hence,
$$P \in \bigcup_{s(A) \cap \text{Pic}^0_X} \tilde{\Sigma}\otimes f_1^*M^{-1}.$$
Since $s(A) \cap \text{Pic}^0_X$ is finite, we see that there are only finitely many choices for $P$.  It follows that the restriction of $p_{12}$ to  $\text{im}(\varepsilon)$ has finite kernel, so that $\text{dim}(A) \leq \text{dim}(B)$.  Applying the same argument to $\Phi^{-1}_{\mathscr{E}}$ we obtain the reverse inequality, so that $\text{dim}(A)=\text{dim}(B)$.

 Next we show that $A$ and $B$ are isogenous.  We have shown thus far that the $\text{im}(\varepsilon) \subseteq A \times B \times \hat{A}$ is an abelian subvariety is isogenous to $A \times B$ via (the restriction of) the projection map $p_{12}:\text{im}(\varepsilon) \to A \times B$.  We claim that the projection map $p_{13}:A \times B \times \hat{A} \to A \times \hat{A}$ restricts to an isogeny $\text{im}(\varepsilon) \to A \times \hat{A}$.  For dimension reasons, it suffices to show that this map is surjective.  Let $(a,P) \in A \times \hat{A}$.  Since $f_1^*$ is surjective, we can find $L \in \text{Pic}^0(X)$ such that $f_1^*L \simeq P$.  We can also find $a' \in A$ such that $q_X(s(a'))=a$.  We have that $p_{13}(\varepsilon(a',L))=(a,P)$, as desired.  It follows that $A \times \hat{A}$ and $A \times B$ are isogenous.  Applying the same argument to $\Phi_{\mathscr{E}}^{-1}$, we see that $B \times \hat{B}$ and $A \times B$ are isogenous.  It follows that $A \times \hat{A}$ and $B \times \hat{B}$ are isogenous so that $A$ and $B$ are isogenous by the argument of \cite[Lemma 2.1]{Hon18}.

Now consider the exact sequences
$$0\to \text{ker}(\mathfrak{b}) \to \text{Pic}^0_X \to B \to 0$$
$$0 \to \text{ker}(\mathfrak{a}) \to \text{Pic}^0_Y \to A \to 0$$
Since $\Psi$ induces an isomorphism $\text{ker}(\mathfrak{b}) \simeq \text{ker}(\mathfrak{a})$, and since $A$ and $B$ are isogenous, we must have that $\text{Pic}^0_X$ and $\text{Pic}^0_Y$ are isogenous.
\end{proof}
 Arguing as in \cite{Ros09}, we can derive the following corollary of \Cref{t}.
\begin{thm}\label{t2}
    Let $X$ and $Y$ be smooth projective varieties that are twisted derived equivalent.  If $Y$ is an abelian variety, then so is $X$.  Moreover, $X$ is isogenous to $Y$.
\end{thm}
\begin{proof}
Let $\mathscr{X} \to X$ and $\mathscr{Y} \to Y$ be $\mathbb{G}_m$-gerbes for which there is an equivalence $\text{D}^b(X,[\mathscr{X}])\simeq \text{D}^b(Y,[\mathscr{Y}])$.  Then, arguing exactly as in \cite[Lemma 2.5]{Lan24}, we obtain that $\text{dim}(X)=\text{dim}(Y)=n$.  By \Cref{t: Olsson}, we have that $\text{Aut}^0_{\mathscr{X}}\simeq \text{Aut}^0_{\mathscr{Y}}$, so that $\text{Aut}^0_{\mathscr{X}}$ is an abelian variety of dimension $2n$.  By \Cref{t}, $\text{Pic}^0_X$ is an $n$-dimensional abelian variety.  Using the exact sequence
$$0 \to \text{Pic}^0_X \to \text{Aut}^0_{\mathscr{X}} \to \text{Aut}^0_X \to 0,$$
we see that $\text{Aut}^0_X$ is also an $n$-dimensional abelian variety.  Consider the action of $\text{Aut}^0_X$ on $X$ by automorphisms, and choose $x \in X$.  By \cite[Lemma 3.3]{Ros09}, the stabilizer of $x$ for this action is finite.  It follows that the orbit map $\tau_x:\text{Aut}^0_X \to X$ is surjective.  By \cite[Lemm 1.13]{Lan24} $X$ is an abelian variety.  It follows immediately from \Cref{t} that $X$ and $Y$ are isogenous.
\end{proof}
\section{A Twisted Derived Equivalence Criterion}
Let $X$ be an abelian variety.  In this section we study semi-homogeneous twisted vector bundles on ($\mathbb{G}_m$-gerbes over) $X$ and homogeneous projective bundles over $X$.  We begin by recalling the dictionary between twisted vector bundles and projective bundles.  After introducing the relevant definitions, we show that this restricts to a correspondence between semi-homogeneous twisted vector bundles and homogeneous projective bundles. The latter were thoroughly studied in \cite{Bri13}.  We use the results of \cite{Bri13} to establish the properties of semi-homogeneous twisted vector bundles needed to study twisted derived equivalences between abelian varieties.

\subsection{Perspectives on the Brauer Group}  Let $X$ be a smooth projective variety.  Let $\pi:\mathscr{X} \to X$ be a $\mathbb{G}_m$-gerbe, and let $\mathscr{E}$ be a $\mathscr{X}$-twisted vector bundle of rank $n$.  We will show that $\mathbb{P}(\mathscr{E})$ is a $\mathbb{G}_m$-gerbe over a $\mathbb{P}^{n-1}$-bundle $P \to X$ whose associated Azumaya algebra is $\pi_*\mathcal{E}nd(\mathscr{E})$.  Moreover, we show that any $\mathbb{P}^{n-1}$-bundle over $X$ arises in this way.  This is surely well-known, but it does not appear to have been written down in precisely this way.

From the exact sequence
$$0 \to \mathbb{G}_m \to GL_n \to PGL_n \to 0,$$
one obtains a long exact sequence on étale cohomology
$$\dots \to H^1(X,\mathbb{G}_m) \to H^1(X,GL_n) \to H^1(X,PGL_n) \xrightarrow{\delta} H^2(X,\mathbb{G}_m).$$
The group $H^1(X,PGL_n)$ can be interpreted as either the set of isomorphism classes of degree $n$ Azumaya algebras over $X$ or the set of isomorphism classes of $\mathbb{P}^{n-1}$-bundles over $X$.  The map $\delta$ sends an object parameterized by $H^1(X,PGL_n)$ to (the isomorphism class of) its \textit{gerbe of trivializations}.

Let $P \to X$ be a $\mathbb{P}^{n-1}$-bundle; we let $\mathscr{X}_P$ denote its gerbe of trivializations.  Following \cite[Section 6]{BS21}, $\pi:\mathscr{X}_P \to X$ is the stack whose objects over a scheme $T \to X$ are pairs $(\mathcal{V},v)$, where $\mathcal{V}$ is a locally free sheaf on $X_T$ and $v:P_T \to \mathbb{P}(\mathcal{V})$ is an isomorphism over $T$.  The banding by $\mathbb{G}_m$ given by the action of $\mathbb{G}_m(T)$ on $\mathcal{V}$ by scalar multiplication.  The gerbe $\mathscr{X}_P$ is endowed with a tautological rank $n$ vector bundle $\mathscr{E}_P$ obtained by forgetting the isomorphism $v$, and it is straightforward to see that $\mathscr{E}_P$ is a $\mathscr{X}$-twisted vector bundle making the following square $2$-Cartesian.
$$
\begin{tikzcd}
{\mathbb{P}(\mathscr{E}_P)} \arrow[r] \arrow[d] & {P} \arrow[d, "p"] \\
{\mathscr{X}} \arrow[r, swap, "\pi"]           & {X}          
\end{tikzcd}.$$
In particular, we see that $\mathbb{P}(\mathscr{E}_P)$ is a $\mathbb{G}_m$-gerbe over $P$.

Let $\mathcal{A}$ be a degree $n$ Azumaya algebra on $X$.  Following \cite[Section 12]{Ols16} the gerbe of trivializations $\mathscr{X}_{\mathcal{A}}$ is the stack whose objects over a morphism $f:T \to X$ are pairs $(\mathcal{V},\lambda)$, where $\mathcal{V}$ is a vector bundle on $T$ and $\lambda:\mathcal{E}nd(\mathcal{V}) \to f^*\mathcal{A}$ is an isomorphism of $\mathcal{O}_T$-algebras.  The banding by $\mathbb{G}_m$ is given by the action of $\mathbb{G}_m(T)$ on $\mathcal{V}$ by scalar multiplication.  $\mathscr{X}_{\mathcal{A}}$ is equipped with a tautological rank $n$ $\mathscr{X}_\mathcal{A}$-twisted vector bundle, which we denote by $\mathscr{E}_\mathcal{A}$.  We will need the following result from \cite{Ols16}.

\begin{prop}[{\cite[Proposition 12.3.11]{Ols16}}]\label{3: p Olsson}
    Let $\pi:\mathscr{X} \to X$ be a $\mathbb{G}_m$-gerbe, and let $\mathscr{E}$ be a $\mathscr{X}$-twisted vector bundle of rank $n$.  Then $\mathcal{A}:=\pi_*\mathcal{E}nd(\mathscr{E})$ is a degree $n$ Azumaya algebra on $X$, and there exists an isomorphism of $\mathbb{G}_m$-gerbes $\phi:\mathscr{X} \to \mathscr{X}_\mathcal{A}$ such that $\mathscr{E} \simeq \phi^*\mathscr{E}_{\mathcal{A}}$.
\end{prop}
Our main result is the following corollary.
\begin{cor}\label{3: c pbs and vbs}
    Let $\pi:\mathscr{X} \to X$ be a $\mathbb{G}_m$-gerbe, and let $\mathscr{E}$ be a $\mathscr{X}$-twisted vector bundle of rank $n$.  Then there exists a $\mathbb{P}^{n-1}$-bundle $P \to X$ such that
    \begin{enumerate}
        \item there exists an isomorphism of $\mathbb{G}_m$-gerbes $\phi:\mathscr{X} \to \mathscr{X}_{P}$ such that $\mathscr{E} \simeq \phi^*\mathscr{E}_P$,
        \item there is a map $\mathbb{P}(\mathscr{E}) \to P$ making the following square $2$-Cartesian
        $$
\begin{tikzcd}
{\mathbb{P}(\mathscr{E})} \arrow[r] \arrow[d] & {P} \arrow[d, "p"] \\
{\mathscr{X}} \arrow[r, swap, "\pi"]           & {X}          
\end{tikzcd},$$
        \item the Azumaya algebra corresponding to $P$ is $\pi_*\mathcal{E}nd(\mathscr{E})$.
    \end{enumerate}
\end{cor}
\begin{proof}
    Let $\mathscr{X}$ and $\mathscr{E}$ be as in the statement, let $\mathcal{A}=\pi_*\mathcal{E}nd(\mathscr{E})$, and let $P$ be the $\mathbb{P}^{n-1}$-bundle associated to $\mathcal{A}$.  By \Cref{3: p Olsson} and the above discussion of $\mathscr{X}_P$, it suffices to show that there is an isomorphism $\psi:\mathscr{X}_\mathcal{A} \xrightarrow{\sim}\mathscr{X}_P$ such that $\psi^*\mathscr{E}_P \simeq \mathscr{E}_\mathcal{A}$.  We construct the desired morphism of stacks as follows.  Given a morphism $f:T \to X$, and a $T$-point $(\mathcal{V}, \lambda)$ of $\mathscr{X}_\mathcal{A}$, we obtain an isomorphism of $PGL_n$-torsors $\text{Isom}(M_n(\mathcal{O}_T),\mathcal{E}nd(\mathcal{V})) \to \text{Isom}(M_n(\mathcal{O}_T),f^*\mathcal{A}))$.  We send $(\mathcal{V},\lambda)$ to the pair $(\mathcal{V},\lambda^*)$, where $\lambda^*$ is the isomorphism of associated $\mathbb{P}^{n-1}$-bundles $\mathbb{P}(\mathcal{V}) \xrightarrow{\sim} P_T$.  One can verify that this indeed defines a morphism of stacks.  It is clear from the construction that $\phi^*\mathscr{E}_\mathcal{A} \simeq \mathcal{E}_P$.
\end{proof}
\subsection{Semi-Homogeneous Bundles}

Let $\pi:\mathscr{X} \to X$ be a $\mathbb{G}_m$-gerbe over $X$.  Recall that there is a short exact sequence
$$0 \to \hat{X} \to \text{Aut}^0_{\mathscr{X}} \xrightarrow{q_X} X \to 0,$$
where $q_X$ sends an automorphism $\sigma:\mathscr{X} \to \mathscr{X}$ to its rigidification.

\begin{defn}
    Let $\mathscr{E}$ be an $\mathscr{X}$-twisted vector bundle.  We say that $\mathscr{E}$ is \textit{semi-homogeneous} if for any $x \in X$, there exists a lift $\tau_x:\mathscr{X} \to \mathscr{X}$ of $t_x$ such that $\tau_x^*\mathscr{E} \simeq \mathscr{E} \otimes \mathscr{L}$ for some line bundle $\mathscr{L}$ on $\mathscr{X}$.
\end{defn}
\begin{rem}
    By a lift of $t_x$, we mean an automorphism $\tau_x \in \text{Aut}^0(\mathscr{X})$ such that $q_X(\tau_x)=t_x$.
\end{rem}

We will primarily be interested in studying those $\mathscr{X}$-twisted vector bundles $\mathscr{E}$ which are $\textit{simple}$, meaning that the natural map $k \to \text{End}(\mathscr{E})$ is an isomorphism.  Let $\mathcal{S}pl_{\mathscr{X}}$ denote the stack of simple $\mathscr{X}$-twisted coherent sheaves.  It is a $\mathbb{G}_m$-gerbe over an algebraic space which we denote by $\text{Spl}_{\mathscr{X}}$.  There is an action of $\hat{X}$ on $\text{Spl}_{\mathscr{X}}$ by the tensor product.  For a simple $\mathscr{X}$-twisted vector bundle $\mathscr{E}$, let $\Sigma(\mathscr{E}) \subset \mathscr{X}$ denote the stabilizer of $\mathscr{E}$ for the $\hat{X}$ action.  $\Sigma(\mathscr{E})$ is a finite closed subgroup scheme of $\hat{X}$, and the orbit map $o_{\mathscr{E}}:\hat{X} \to \text{Spl}_{\mathscr{X}}$ factors through a map $\widetilde{o_\mathscr{E}}:\hat{X}/\Sigma(\mathscr{E}) \to \text{Spl}_{\mathscr{X}}$.

As we noted above, any twisted vector bundle arises from a projective bundle.  We are interested in determining what kinds of projective bundles give rise to simple semi-homogeneous vector bundles.

\begin{defn}
    A projective bundle $P \to X$ is \textit{homogeneous} if $t_x^*P \simeq P$ for all $x \in X$.
\end{defn}
\begin{lem}\label{l homogoneity}
    A projective bundle $P \to X$ is homogeneous if and only if the $\mathscr{X}_P$-twisted vector bundle $\mathscr{E}_P$ is semi-homogeneous.
\end{lem}
\begin{proof}
    Let $x \in X(k)$, and let $\tau_x:\mathscr{X}_P \to \mathscr{X}_P$ be a lift of $t_x$.  If $\mathscr{E}_P$ is semi-homogeneous, then $\mathbb{P}(\mathscr{E}_P) \simeq \mathbb{P}(\tau_x^*\mathscr{E}_P) \simeq \tau_x^*\mathbb{P}(\mathscr{E}_P)$.  Rigidifying along $\mathbb{G}_m$, we get $P \simeq t_x^*P$.  If $P$ is homogeneous, then $P \simeq t_x^*P$ so that
    $$\mathbb{P}(\mathscr{E}_P) \simeq \pi^*P \simeq \pi^*t_x^*P \simeq \tau_x^*\pi^*P \simeq \tau_x^*\mathbb{P}(\mathscr{E}_P)\simeq \mathbb{P}(\tau_x^*\mathscr{E}_P).$$
    It follows that $\mathscr{E}_P \simeq \tau_x^*\mathscr{E}_P \otimes \mathscr{L}$ for some line bundle $\mathscr{L}$ on $\mathscr{X}_P$.
\end{proof}
We will restrict our attention to those homogeneous projective bundles which are $\textit{irreducible}$.  Before giving the definition, recall that with any projective bundle $P \to X$ we may associate the $\textit{adjoint bundle}$ $\text{ad}(P)$.  If we let $\mathcal{A}_P$ denote the Azumaya algebra corresponding to $P$, then is the quotient of $\mathcal{A}_P$ by $\mathcal{O}_X$ by \cite[Section 2]{Bri13}.
\begin{defn}\cite[Proposition 3.5]{Bri13}  A homogeneous projective bundle $P$ is called $\textit{irreducible}$ if $H^0(X,\text{ad}(P))=0$.
\end{defn}
Irreducible homogeneous projective bundles have been classified in \cite{Bri13}.

\begin{prop}\label{p classification}\cite[Proposition 3.1]{Bri13}
\begin{enumerate}
    \item The irreducible homogeneous $\mathbb{P}^{n-1}$-bundles are classified by pairs $(H,e)$, where $H \subset X[n]$ is a subgroup of order $n^2$ and $e:H \times H \to \mathbb{G}_m$ is a non-degenerate alternating pairing.
    \item  For the bundle $P$ corresponding to $(H,e)$, the Azumaya algebra $\mathcal{A}_P$ admits a grading by the group $H$, namely
    $$\mathcal{A}_P \simeq \bigoplus_{L \in H} L.$$
\end{enumerate}
\end{prop}

\begin{lem}\label{3: l finite groups}
    Let $H$ be a finite group scheme.  The following are equivalent:
    \begin{enumerate}
        \item There exists a non-degenerate alternating pairing $e:H \times H \to \mathbb{G}_m$
        \item $H \simeq \bigoplus_{i=1}^r (\mathbb{Z}/m_i\mathbb{Z})^2$.
    \end{enumerate}
\end{lem}
\begin{proof}
If $H \simeq \oplus_{i=1}^r(\mathbb{Z}/m_i\mathbb{Z}^2)$, then since we are working over $\mathbb{C}$, we have that $H \simeq K \oplus \hat{K}$, where $K=\oplus_{i=1}^r(\mathbb{Z}/m_i\mathbb{Z})$.  The Heisenberg group is a central extension of $H$ by $\mathbb{G}_m$ whose commutator pairing has the desired properties (see \cite[(8.21)]{EvdGM}).

Let $e:H \times H \to \mathbb{G}_m$ denote a non-degenerate alternating pairing.  By the discussion preceeding \cite[Theorem 4.3]{Pol96}, there exists a non-degenerate central extension of $H$ by $\mathbb{G}_m$ for which $e$ is the commutator pairing.  The result follows from \cite[Lemma 8.24]{EvdGM} and the fact that finite group schemes over $\mathbb{C}$ are constant.

\end{proof}
For an irreducible homogeneous bundle $P$, let $(H_P,e_P)$ denote the corresponding pair as in \Cref{p classification}. 

\begin{lem}\label{l simple}
    Let $P$ be a homogeneous projective bundle.  Then $P$ is irreducible if and only if $\mathscr{E}_P$ simple.
\end{lem}
\begin{proof}
By \Cref{3: c pbs and vbs}(3), $H^0(X,\text{ad}(P)) \simeq \text{End}(\mathscr{E}_P)/k$.  This is $0$ if and only if $\mathscr{E}_P$ is simple.
\end{proof}
\begin{cor}\label{c orthog}
    Let $P$ be an irreducible homogeneous projective bundle, and let $L \in \text{Pic}^0(X)$.  If $L \in H_P$, then $\mathscr{E}_P \simeq \mathscr{E}_P \otimes \pi^*L$.  Otherwise $\text{Ext}^i(\mathscr{E}_P,\mathscr{E}_P \otimes \pi^*L)=0$ for all $i$.  In particular, $H_P \simeq \Sigma(\mathscr{E}_P)$ as subgroups of $\hat{X}$.
\end{cor}
\begin{proof}
    Recall that there is an isomorphism
    \begin{equation}
        \mathcal{E}nd(\mathscr{E}_P) \simeq \bigoplus_{L \in H_P}L.
    \end{equation}
If $L \in H_P$, then $\text{dim}(\text{Hom}(\mathscr{E}_P,\mathscr{E}_P \otimes \pi^*L))=1$.  Let $\phi:\mathscr{E}_P \to \mathscr{E}_P \otimes \pi^*L$ be a nonzero morphism, and let $\mathscr{K}=\text{ker}(\phi)$.  For any $x \in X$, we can choose a lift $\tau_x:\mathscr{X}_P \to \mathscr{X}_P$ of $t_x$ such that there exists an isomorphism $\lambda:\tau_x^*\mathscr{E}_P \xrightarrow{\sim} \mathscr{E}_P \otimes \mathscr{L}$.  We can also choose an isomorphism $\tau_x^*(\mathscr{E}_P \otimes \pi^*L) \simeq (\mathscr{E}_P \otimes \pi^*L)\otimes \mathscr{L}$.  Rescaling $\lambda$ if necessary, we can make the following diagram commute
$$
\begin{tikzcd}
{\tau_x^*\mathscr{E}_P} \arrow[r, "\tau_x^*\phi"] \arrow[d, swap, "\lambda"] & {\tau_x^*(\mathscr{E}_P \otimes \pi^*L)} \arrow[d, "\rotatebox{90}{$\sim$}"] \\
{\mathscr{E}_P \otimes \mathscr{L}} \arrow[r, swap, "\phi \otimes \mathscr{L}"]           & {(\mathscr{E}_P \otimes \pi^*L)\otimes \mathscr{L}}          
\end{tikzcd}.$$
It follows that for any $x \in X$, $\tau_x^*\mathscr{K} \simeq \mathscr{K} \otimes \mathscr{L}$.  Arguing as in \Cref{l: vector bundle}, we see that $\mathscr{K}$ is locally free.  The vector bundle $\mathcal{H}om(\mathscr{K},\mathscr{E}_P)$ satisfies $t_x^*\mathcal{H}om(\mathscr{K},\mathscr{E}_P) \simeq \mathcal{H}om(\mathscr{K},\mathscr{E}_P)$, which means that it is a homogeneous vector bundle (see \cite[Definition 4.4]{Mu78}).  By \cite[Proposition 4.18]{Mu78}, there exists a nonzero morphism $\psi:\mathscr{E}_P \to \mathscr{K}$.  Composing it with the inclusion $\mathscr{K} \subseteq \mathscr{E}_P$, we obtain a nonzero endomorphism of $\mathscr{E}_P$, which must be multiplication by a scalar.  It follows that either $\mathscr{K}\simeq 0$, or $\mathscr{K} \simeq \mathscr{E}_P$.  The latter cannot happen since $\phi$ is nonzero.  Thus, $\phi$ is an isomorphism.
The second statement follows from tensoring (1) by $L$ and taking cohomology.  The third statement is clear.
\end{proof}
Finally, we are able to prove our main result:

\begin{thm}\label{t criterion}
    Let $X$ be an abelian variety, and let $f:\hat{X} \to Y$ whose kernel $H$ is isomorphic to $\oplus_{i=1}^r(\mathbb{Z}/m_i\mathbb{Z})^2$.  Then $X$ and $Y$ are twisted derived equivalent with vector bundle kernel.
\end{thm}
\begin{proof}
    By \Cref{3: l finite groups} and \Cref{p classification}, there exists an irreducible homogeneous projective bundle $P \to X$ such that $H_P = H$.  Consider the orbit map $\widetilde{o_{\mathscr{E}_P}}:\hat{X}/\Sigma(\mathscr{E}_P) \to \text{Spl}_{\mathscr{X}_P}$.  By \Cref{c orthog}, $\Sigma(\mathscr{E}_P) \simeq H$ as subgroups of $\hat{X}$, so we obtain a map $o:Y \to \text{Spl}_{\mathscr{X}_P}$.  The map $o$ corresponds to a twisted sheaf $\mathscr{U}$ on $\mathscr{Y} \times \mathscr{X}_P$, where $\pi_Y:\mathscr{Y} \to Y$ is the pullback of the gerbe $\mathcal{S}pl_{\mathscr{X}_P}$ along $o$.  Let $\alpha \in H^2(X,\mathbb{G}_m)$ and $\beta \in H^2(Y,\mathbb{G}_m)$ denote the cohomology classes of $\mathscr{X}_P$ and $\mathscr{Y}_P$ respectively, and consider the Fourier-Mukai transform $\Phi_{\mathscr{U}}:\text{D}^b(Y,\beta^{-1}) \to \text{D}^b(Y,\alpha)$.  To see that $\Phi_{\mathscr{U}}$ is an equivalence, it suffices to show that it satisfies the conditions of \cite[Theorem 3.2.1]{C00}.  The final condition is trivially satisfied since $\omega_X\simeq \mathcal{O}_X$, so it remains to show that $\Phi_{\mathscr{U}}$ is fully faithful.  This amounts to verifying the following conditions:
    \begin{enumerate}
        \item For each $y \in \mathscr{Y}(k)$, $\text{End}(\mathscr{U}_y)=k$
        \item For each pair of points $y_1,y_2 \in \mathscr{Y}(k)$, $\text{Ext}^i(\mathscr{U}_{y_1},\mathscr{U}_{y_2})=0$ unless $y_1 \simeq y_2$ and $0 \leq i \leq \text{dim}(Y)$.
    \end{enumerate}
Let $y \in \mathscr{Y}(k)$.  Then $\mathscr{U}_y \simeq \mathscr{E}_P \otimes \pi^*L$, where $f(L)=\pi_Y(y)$.  The first condition follows from \Cref{l simple}, and the second follows from \Cref{c orthog}.  Since $\mathscr{U}_y$ is a vector bundle for each $y\in \mathscr{Y}(k)$, $\mathscr{U}$ is a vector bundle.
\end{proof}
We prove the following partial converse to \Cref{t criterion}
\begin{thm}
    Let $\Phi_\mathscr{U}:\text{D}^b(X,\alpha) \to \text{D}^b(Y,\beta)$ be a twisted derived equivalence between two abelian varieties.  If $\mathscr{U}$ is locally free, then there exists an isogeny $f:\hat{X} \to Y$ whose kernel is isomorphic to $\oplus_{i=1}^r(\mathbb{Z}/m_i\mathbb{Z})^2$.
\end{thm}
\begin{proof}
    Let $\pi_X:\mathscr{X} \to X$ and $\pi_Y:\mathscr{Y} \to Y$ denote the $\mathbb{G}_m$-gerbes corresponding to $\alpha$ and $\beta$ respectively.  Let $x \in \mathscr{X}(k)$, let $k(x)$ denote the twisted skyscraper sheaf supported at $x$, and let $\mathscr{E}=\mathscr{U}_x=\Phi_{\mathscr{U}}(k(x))$.  By \Cref{t: Olsson}, we have a diagram
    $$
\begin{tikzcd}
{0} \arrow[r] & {\hat{X}} \arrow[r] & {\text{Aut}^0_{\mathscr{X}}} \arrow[r, "q_X"] \arrow[d, "\rotatebox{90}{$\sim$}"] & {X} \arrow[r] & {0} \\
{0} \arrow[r] & {\hat{Y}} \arrow[r] & {\text{Aut}^0_{\mathscr{Y}}} \arrow[r, "q_Y"]           & {Y} \arrow[r] & {0}
\end{tikzcd}$$
in which the vertical arrow is an isomorphism.  From this we obtain a morphism $f:\hat{X} \to Y$.  Since $\text{Pic}^0(X)=\{\sigma \in \text{Aut}^0(\mathscr{X}): \sigma^*k(x)\simeq k(x)\}$, \Cref{l: kernels} implies that the image of $\text{Pic}^0(X)$ is $\{\sigma \in \text{Aut}^0(\mathscr{Y}): \sigma^*\mathscr{E} \simeq \mathscr{E}\}$.  Therefore the $k$-points of $H:=\text{ker}(f)$ are precisely those $L \in \text{Pic}^0(Y)$ such that $\mathscr{E}\otimes \pi_Y^*L \simeq \mathscr{E}$.  Thus $H \simeq \Sigma(\mathscr{E})$ is finite, and $f$ is an isogeny.  By \Cref{3: c pbs and vbs}(1), there exists a projective bundle $Q \to Y$ such that $\mathscr{Y} \simeq \mathscr{Y}_Q$ and $\mathscr{E} \simeq \mathscr{E}_Q$.  Surjectivity of $f$ implies that $\mathscr{E}$ is semi-homogeneous so that, by \Cref{l homogoneity}, $Q$ is homogeneous.  Since $\mathscr{E}$ is simple, \Cref{l simple} implies that $Q$ is irreducible.  By \Cref{c orthog}, we have isomorphisms $H \simeq \Sigma(\mathscr{E}) \simeq \Sigma(\mathscr{E}_Q) \simeq H_Q$, where $H_Q$ is as in \Cref{p classification}.  It follows from \Cref{3: l finite groups} that $H$ has the desired properties.
\end{proof}
Putting everything together, we have:
\begin{cor}\label{3: c main}
    Let $X$ and $Y$ be abelian varieties.  Then the following are equivalent.
    \begin{enumerate}
        \item There exists an isogeny $f:\hat{X} \to Y$ whose kernel is isomorphic to $\oplus_{i=1}^r (\mathbb{Z}/m_i\mathbb{Z})^2$.
        \item There exists a twisted derived equivalence between $X$ and $Y$ with vector bundle kernel.
    \end{enumerate}
\end{cor}
\section{Symplectic Isomorphisms}
In this section, we complete the proof of \Cref{i t: main} by proving the following result.

\begin{thm}\label{4: t main}
    Let $X$ and $Y$ be complex abelian varieties.  If $X$ and $Y$ are twisted derived equivalent, then there exists an isogeny $f:\hat{X} \to Y$ whose kernel is isomorphic to $\oplus_{i=1}^r (\mathbb{Z}/m_i\mathbb{Z})^2$. 
\end{thm}
We prove this using the language of symplectic abelian varieties, which we recall below.  Using this language, we give a general outline of the proof.  To produce an isogeny as in \Cref{4: t main}, it suffices find a symplectic abelian variety $A$ satisfying the following conditions
\begin{enumerate}
    \item $A$ contains $\hat{X}$ and $\hat{Y}$ as Lagrangians,
    \item $\hat{X} \cap \hat{Y}$ is finite
\end{enumerate}
(see \Cref{4 l: finite intersections}).  \cite[Theorem 1.1]{LLT25} allows us to find $A$ satisfying condition $(1)$.  While $\hat{X} \cap \hat{Y}$ may not be finite, we show that there exists another embedding $\iota:\hat{Y} \hookrightarrow A$ witnessing $\hat{Y}$ as a Lagrangian with finite intersection with $\hat{X}$ (see \Cref{4: p finiteness}).  In the case where $A=Y \times \hat{Y}$, the existence of such an $\iota$ was established in \cite[Lemma 2.2.7(ii)]{Pol12}.  In general, $A$ is only isogenous to $Y \times \hat{Y}$; however, after establishing some basic lemmas on what we call \textit{symplectic isogenies}, we can obtain the existence of $\iota$ as a corollary of \cite[Lemma 2.2.7(ii)]{Pol12}.

\subsection{Symplectic Abelian Varieties}
Let $A$ be an abelian variety.  A \textit{biextension} of $A^2$ (by $\mathbb{G}_m$) is line bundle $L_A$ on $A^2$ together with rigidifcations
$$\varepsilon:L_A\mid_{A \times 0} \xrightarrow{\sim} \mathcal{O}_A,\quad \delta:L_A\mid_{0 \times A} \xrightarrow{\sim} \mathcal{O}_A$$
whose restrictions to $0 \times 0$ coincide.  Note that $L_A$ determines a homomorphism $\psi_{L_A}:A \to \hat{A}$ and that $L_A$ is uniquely determined by $\psi_{L_A}$.  We say that the biextension $L_A$ is \textit{symplectic} if $\psi_{L_A}$ is an isomorphism satisfying $\widehat{\psi_{L_A}}=-\psi_{L_A}$.

\begin{defn}
    A \textit{symplectic abelian variety} is a pair $(A,L_A)$, where $A$ is an abelian variety and $L_A$ is a symplectic biextension of $A^2$.
\end{defn}
When no confusion arises, will often simply write $A$ for a symplectic abelian variety, and we will write $\psi_A$ instead of $\psi_{L_A}$.
\begin{defn}
    An abelian subvariety $Z$ of a symplectic abelian variety $A$ is called \textit{isotropic} if the composition
    $$Z \xrightarrow{i} A \xrightarrow{\hat{i}\psi_A} \hat{Z}$$
    is zero, where $i:Z \hookrightarrow A$ is the inclusion.  An isotropic abelian subvariety is \textit{Lagrangian} if the above sequence is exact.
\end{defn}
Note that if $Z \subset A$ is Lagrangian, then there is an isomorphism $A/Z \simeq \hat{Z}$.  If $W$ and $Z$ are Lagrangian abelian subvarieties of $A$ whose intersection is finite, then the quotient map restricts to an isogeny $W \to \hat{Z}$.  Moreover, by \cite{Pol96}, there is a non-degenerate alternating pairing $(Z \cap W) \times (Z \cap W) \to \mathbb{G}_m$.  Applying \Cref{3: l finite groups}, we obtain the following.

\begin{lem}\label{4 l: finite intersections}
    Let $A$ be a symplectic abelian variety, and let $Z,W \subset A$ be Lagrangian abelian subvarieties.  Then there is an isogeny $W \to \hat{Z}$ whose kernel is isomorphic to $\bigoplus_{i=1}^r (\mathbb{Z}/m_i\mathbb{Z})^2$.
\end{lem}
\subsection{Symplectic Isogenies}
Let $(A,L_A)$ and $(B,L_B)$ be symplectic abelian varieties.  We say that an isomorphism $f:A \to B$ is \textit{symplectic} (or \textit{isometric}) if either one of the following equivalent conditions hold:
\begin{enumerate}
    \item $L_A \simeq (f \times f)^*L_B$
    \item $\psi_{L_A}=\hat{f}\psi_{L_B} f.$
\end{enumerate}
The following more general class of morphisms arises naturally when studying symplectic abelian varieties.
\begin{defn}
    Let $A$ and $B$ be symplectic abelian varieties.  A \textit{symplectic isogeny} between $A$ and $B$ is an isogeny $f:A \to B$ such that $n\psi_A=\hat{f}\psi_Bf$ for some $n \in \mathbb{Z}$.
\end{defn}
We will need the following two lemmas.

\begin{lem}\label{4: l lagrangians}
    Let $f:A \to B$ be a symplectic isogeny, let $Z \subset B$ be an abelian subvariety, and let $Z'\subset A$ denote the neutral component of $f^{-1}(Z)$.  Then $Z$ is Lagrangian if and only if $Z'$ is.
\end{lem}
\begin{proof}
    Let $i:Z \hookrightarrow B$ (resp. $j:Z' \hookrightarrow A$) denote the inclusion of $Z$ (resp. $Z'$), and let $g:Z' \to Z$ denote the isogeny induced by $f$.  We have a commutative diagram
    $$
\begin{tikzcd}
{Z'} \arrow[r, "j"] \arrow[d, swap, "g"] & {A} \arrow[d, swap, "f"] \arrow[r, "n\psi_A"] & {\hat{A}} \arrow[r, "\hat{j}"]           & {\widehat{Z'}}           \\
{Z} \arrow[r, swap, "i"]           & {B} \arrow[r, swap, "\psi_B"]           & {\hat{B}} \arrow[u,  swap, "\hat{f}"] \arrow[r, swap, "\hat{i}"] & {\hat{Z}} \arrow[u, swap, "\hat{g}"]
\end{tikzcd}.$$
If $Z$ is Lagrangian, then $i\psi_B\hat{i}=0$ so that $n(\hat{j}\psi_Aj)=0$.  It follows that $\hat{j}\psi_Bj=0$.  By \cite[Lemma 2.2.3(ii)]{Pol12}, we have that $Z'$ is Lagrangian.

If $Z'$ is Lagrangian, then we have that $\hat{g}(\hat{i}\psi_Bi)g=0$, so that $\hat{i}\psi_Bi=0$.  By \cite[Lemma 2.2.3(ii)]{Pol12}, we have that $Z$ is Lagrangian.
\end{proof}
\begin{lem}\label{4: l finiteness}
    Let $f:A \to B$ be an isogeny, let $Z$ and $W$ be abelian subvarieties of $B$, and let $Z'$ $($resp. $W')$ denote the neutral component of $f^{-1}(Z)$ $($resp. $f^{-1}(W))$.  If $Z'\cap W'$ is finite, then so is $Z \cap W$.
\end{lem}
\begin{proof}
    We have that $Z' \to Z$ and $W' \to W$ are surjective, so that $Z'\times_B W' \to Z \cap W$ is surjective.  Thus, it suffices to show that $Z' \times_B W'$ is finite.  The map $Z'\times_B W' \to \text{ker}(f)$, given by $(z,w) \mapsto z-w$, fits into an exact sequence
    $$0 \to Z'\cap W' \to Z'\times_B W' \to \text{ker}(f).$$
    Since $Z'\cap W'$ and $\text{ker}(f)$ are finite, so is $Z'\times_B W'$.
\end{proof}
\subsection{Symplectic AVs associated to Brauer classes}
Let $X$ be an abelian variety, and let $\alpha \in \text{Br}(X)$.  In \cite{LLT25}, they constructed a symplectic abelian variety $A_{(X,\alpha)}$ associated to the pair $(X,\alpha)$.  We begin by recalling the construction.  

Since $\text{Br}(X)$ is torsion, we may view $\alpha$ as an element of $\text{Br}(X)[n]$ for some $n \in \mathbb{Z}$.  The Kummer sequence induces an exact sequence
$$0\to \text{NS}(X)/n \to \text{Hom}(\wedge^2 X[n],\mu_n) \to \text{Br}(X)[n] \to 0.$$
It follows that we may associate with $\alpha$ an alternating pairing $e_{\alpha}:X[n] \times X[n] \to \mathbb{G}_m$ that is not necessarily unique.  Let $\phi_{\alpha}:X[n] \to \hat{X}[n]$ denote the skew-symmetric homomorphism corresponding to $e_{\alpha}$.  Let $K_\alpha$ denote the graph of $\phi_\alpha$, viewed as a subgroup of $X \times \hat{X}$, and let 
$$A_{(X,\alpha)}=(X \times \hat{X})/K_{\alpha}.$$
We give $A_{(X,\alpha)}$ the structure of a symplectic abelian variety as follows.  The biextension $\mathscr{L}_X:=p_{23}^*\mathscr{P} \otimes p_{14}^*\mathscr{P}^{-1}$ makes $X \times \hat{X}$ into a symplectic abelian variety.  By \cite[Section 1]{Pol96}, $\mathscr{L}_X^{\otimes n}$ descends to symplectic biextension $\mathscr{L}_{(X,\alpha)}$ of $A_{(X,\alpha)}$.  By \cite[Section 1]{Pol96}, the symplectic abelian variety $(A_{(X,\alpha)},\mathscr{L}_{(X,\alpha})$ is independent of our choice of $n$ and $e_\alpha$.  However, for the remainder of this section, we will fix our choice of $n$ and $e_\alpha$ used to construct $A_{(X,\alpha)}$.

Let $\pi:X \times \hat{X} \to A_{(X,\alpha)}$ denote the quotient by $K_\alpha$, and let $\psi_\alpha=\psi_{\mathscr{L}_\alpha}$.  As noted in \cite[Section 3.2]{LLT25}, $\psi_\alpha$ is the unique map satisfying $\psi_\alpha=n\hat{\pi}\psi_{\alpha}\pi$.  It follows that $\pi$ is a symplectic isogeny.  We will use this to study Lagrangians in $A_{(X,\alpha)}$.  In particular, we wish to prove the following.

\begin{prop}\label{4: p finiteness}
    Let $Z \subset A_{(X,\alpha)}$ be a Lagrangian.  Then there exists an embedding $\iota:\hat{X} \hookrightarrow A_{(X,\alpha)}$ such that $\iota(\hat{X})$ is a Lagrangian and $Z\cap \iota(\hat{X})$ is finite.
\end{prop}

We will derive this as a corollary of the following lemma of Polishchuk's.

\begin{lem}[{\cite[Lemma 2.2.7(ii)]{Pol96}}]\label{4: l pol} Let $Z \subset X \times \hat{X}$ be a Lagrangian, and let $L$ be an ample line bundle on $\hat{A}$.  Then for almost all $m \in \mathbb{Z}$, the Lagrangian $\Gamma(m\phi_L)$ has finite intersection with $Z$, where $\Gamma(m\phi_L)$ is the graph of $m\phi_L:\hat{A} \to A$.   
\end{lem}

\begin{proof}[Proof of \Cref{4: p finiteness}]
    Let $Z'$ denote the neutral component of $\pi^{-1}(Z)$.  By \Cref{4: l lagrangians}, $Z'$ Lagrangian.  Applying \Cref{4: l pol} to $Z'$, we can find an ample line bundle $L$ on $\hat{X}$ and an integer $m$ divisible by $n$ such that $Z'\cap \Gamma(m\phi_L)$ is finite.
    
    Let $\iota=\pi \circ m\phi_L:\hat{X} \to A_{(X,\alpha)}$.  The condition that $n \mid m$ implies that $\iota$ is injective.  Since $\Gamma(m\phi_L)$ is Lagrangian, \Cref{4: l lagrangians} implies that $\iota(\hat{X})$ is too.  Since $Z' \cap \Gamma(m\phi_L)$ is finite, we can apply \Cref{4: l finiteness} to conclude that $\iota(\hat{X}) \cap Z$ is finite.
\end{proof}
\subsection{Twisted Derived Equivalence}
Before giving the proof of \Cref{4: t main}, we need to recall \cite[Theorem 1.1]{LLT25}.

\begin{thm}[{\cite[Theorem 1.1]{LLT25}}]\label{4: t LLT} Let $X$ and $Y$ be abelian varieties, let $\alpha \in \text{Br}(X)$, and let $\beta \in \text{Br}(Y)$.  If $\Db(X,\alpha) \simeq \Db(Y,\beta)$, then there exists a symplectic isomorphism $A_{(X,\alpha)} \simeq A_{(Y,\beta)}$.
\end{thm}

\begin{proof}[Proof of \Cref{4: t main}]
Let $\Db(X,\alpha) \simeq \Db(Y,\beta)$ be a twisted derived equivalence, and let $f:A_{(X,\alpha)} \xrightarrow{\sim} A_{(Y,\beta)}$ the symplectic isomorphism obtained from \Cref{4: t LLT}.  The result follows from applying \Cref{4: p finiteness} and \Cref{4 l: finite intersections} to the Lagrangian $f(\hat{X}) \subset A_{(Y,\beta)}$.
    
\end{proof}
\printbibliography
\end{document}